\documentclass[12pt]{amsart}
\usepackage{amssymb}
\usepackage{amsmath}
\usepackage{comment}
\usepackage[dvips]{graphicx}
\usepackage{hyperref}
\usepackage[capitalize,nameinlink,noabbrev,nosort]{cleveref}
\usepackage[dvipsnames]{xcolor}
\usepackage{mathtools}
\usepackage{enumitem}
\usepackage[headings]{fullpage}
\usepackage{ytableau}
\usepackage{yhmath}
\usepackage{mathdots}
\usepackage{float}
\usepackage{hyperref,nameref}

\usepackage{tikz}
\usetikzlibrary{arrows}
\usetikzlibrary{decorations.markings}
\usetikzlibrary{tikzmark,decorations.pathreplacing}
\usetikzlibrary{shapes.geometric}
\numberwithin{equation}{section}

\newcommand{\ds}{\displaystyle}

\DeclareMathOperator{\hs}{hs}
\DeclareMathOperator{\s}{s}
\DeclareMathOperator{\spo}{spo}

\renewcommand{\sp}{\mathrm{sp}}
\DeclareMathOperator{\oo}{so}

\renewcommand{\oe}{\mathrm{o}^{\text{even}}}

\newcommand{\x}{\bar{x}}
\newcommand{\X}{\overline{X}}
\def\ov{\overline}

\DeclareMathOperator{\hgt}{ht}

\DeclareMathOperator{\VS}{VS}
\DeclareMathOperator{\BS}{BS}
\DeclareMathOperator{\HST}{HS}

\def\n{\none}
\theoremstyle{plain}
\newtheorem{thm}{Theorem}[section]

\newtheorem{cor}[thm]{Corollary}
\newtheorem{lem}[thm]{Lemma}
\theoremstyle{definition}
\newtheorem{defn}[thm]{Definition}
\newtheorem{rem}[thm]{Remark}

\newtheorem{eg}[thm]{Example}

\crefname{thm}{Theorem}{Theorems}

\title[Murnaghan--Nakayama rules for orthosymplectic Schur functions]{Murnaghan--Nakayama rules for symplectic, orthogonal and orthosymplectic Schur functions} 

\author{Nishu Kumari}

\address{Nishu Kumari, 
The Institute of Mathematical Sciences, Chennai, Tamil Nadu 600113, India.}
\email{nishukumari@alum.iisc.ac.in}

\author{Anna Stokke$^{\ast}$}\thanks{$^\ast$ A. Stokke was supported by an NSERC Discovery Grant.}
\address{Anna Stokke, The University of Winnipeg, 515 Portage Avenue, Winnipeg, Manitoba, Canada.}
\email{a.stokke@uwinnipeg.ca}

\date{}
\begin{document}
\begin{abstract}
We establish new Murnaghan--Nakayama rules for symplectic, orthogonal and orthosymplectic Schur functions.  The classical Murnaghan--Nakayama rule expresses the product of a power sum symmetric function with a Schur function as a linear combination of Schur functions.  Symplectic and orthogonal Schur functions correspond to characters of irreducible representations of symplectic and orthogonal groups.  Orthosymplectic Schur functions arise as characters of orthosymplectic Lie superalgebras and are hybrids of symplectic and ordinary Schur functions. We derive explicit formulas for the product of the relevant  power sum  function with each of these functions, which can partly be described combinatorially using border strip manipulations. Our Murnaghan--Nakayama rules each include three distinct terms: a classical term corresponding to the addition of border strips to the relevant Young diagram, a term involving the removal of border strips, and a third term, which we describe both algebraically and combinatorially.
\end{abstract}
\maketitle
\section{Introduction}

The set of Schur functions, which are indexed by integer partitions, form a basis for the algebra of symmetric functions, so the product of two Schur functions can be expressed uniquely as a linear combination of Schur functions.  The Littlewood--Richardson rule gives a combinatorial description of the coefficients that appear in the expansion.  When one of the Schur functions is indexed by a one-row partition, the resulting combinatorial description is known as Pieri's rule.  

The power sum symmetric functions are another basis of the algebra of symmetric functions and the Murnaghan--Nakayama rule expresses the product of a Schur function  with a power sum symmetric function as a linear combination of Schur functions.  Given a partition $\mu$ and a positive integer $r$,
\begin{equation} \label{eq:mn-rule} p_rs_{\mu}=\sum_{\lambda} (-1)^{\mbox{ht}(\lambda/\mu)}s_\lambda, \end{equation}
where the sum is over all partitions $\lambda \supseteq \mu$, $\lambda/\mu$ is a border strip of length $r$ and $\mbox{ht}(\lambda/\mu)$ is one less than the number of rows in the border strip.  The rule was first proved by Littlewood and Richardson in \cite{littlewood1934} (see Stanely \cite[Pg. 401]{Stanley_2023}), with other proofs given later by Murnaghan \cite{murnaghan} and Nakayma \cite{nakayama}.   Additional proofs are given in \cite{james-1978} and \cite{mendes-2019}.  

Konvalinka gave a quantum version of the Murnaghan--Nakayama rule in \cite{konvalinka-mn} by giving the expansion of a skew Schur function with a $q$-deformation power sum function in terms of skew Schur functions.   Murnaghan--Nakayama rules have been given in several other settings (see for example \cite{jing-liu-mn}, \cite{Nguyen-hiep-mn}, \cite{Benedetti-mn}, \cite{tewari-mn}).

Symplectic and orthogonal Schur polynomials are characters of irreducible representations of symplectic  and orthogonal groups.  They are symmetric Laurent polynomials in the $2n$ variables $x_1^{\pm 1},\ldots,x_n^{\pm 1}$ and can be described combinatorially using symplectic (respectively orthogonal) tableaux. Sundaram gave Pieri rules for symplectic and orthogonal characters in \cite{sundaram-symplectic} and \cite{sundaram-orthogonal}.
Hook Schur polynomials (or supersymmetric Schur polynomials), in $n+m$ commuting variables $x_1,\ldots,x_n,y_1,\ldots,y_m$, are characters of irreducible representations of general linear Lie superalgebras.  Remmel \cite{remmel} proved that hook Schur functions satisfy the same Pieri rule as ordinary Schur functions.  Since hook Schur polynomials can be expressed in terms of Schur polynomials via a plethystic difference operator, a supersymmetric Murnaghan--Nakayama rule follows from the classic Murnaghan--Nakayama rule (see \cref{thm:hook}). 

Orthosymplectic Schur functions, which are characters of modules for orthosymplectic Lie superalgebras, are polynomials in $2n+m$ variables $x_1^{\pm 1},\ldots,x_n^{\pm 1},y_1,\ldots,y_m.$  They can be viewed as hybrids of symplectic Schur functions and ordinary Schur functions (see \cite{benkart-1988}).  A Pieri rule for orthosymplectic Schur functions is given in \cite{stokke-pieri}, where it is shown that the coefficients arising in the expansion  of the product are the same as those that appear in the expansion of the product of the corresponding symplectic Schur functions.  Combinatorial identities for orthosymplectic Schur functions have been given recently in \cite{kumari-orthosymplectic}, \cite{patel-stokke} and \cite{stokke-latticepaths}.

In this paper, we prove symplectic, orthogonal and orthosymplectic Murnaghan--Nakayma rules.  Our symplectic Murnaghan--Nakayama rule, detailed in  \cref{thm:smp}, expresses the product of a symplectic power sum $\overline{p}_r$ with a symplectic Schur polynomial $\sp_{\mu}$ as a linear combination of symplectic Schur polynomials obtained as three distinct sums.  The first sum mirrors \eqref{eq:mn-rule}, and can be found combinatorially by adding border strips of length $r$ to the Young diagram of shape $\mu$, while the second sum comes from removing length $r$ border strips from the Young diagram.  The third sum is more complicated -- we give both an algebraic and combinatorial description in  \cref{sec:s-o-rules}.  The orthogonal Murnaghan--Nakayama rules, given in  \cref{sp-oo} and \cref{sp-eoo} take a similar form.  

Unlike in the Pieri rule setting, the orthosymplectic Murnaghan--Nakayama rule, which we prove in  \cref{spo-MN-m}, does not identically mirror the symplectic Murnaghan--Nakayama rule; the third sum involves both symplectic Schur functions and ordinary Schur functions. 

The paper is organized as follows.  We begin with a preliminary section, \cref{sec:prelim}, where we discuss key symmetric and supersymmetric polynomials, symplectic, orthogonal and orthosymplectic Schur functions, the classic Murnaghan--Nakayama rule and various relevant results.  In \cref{sec:s-o-rules} we prove symplectic and orthogonal Murnaghan--Nakayama rules.  Finally, in \cref{sec:osprule}, we prove an orthosymplectic Murnaghan--Nakayama rule.

\section{Preliminaries}\label{sec:prelim}

In this section, we review relevant definitions and results. For more details, see ~\cite{macdonald-2015} or ~\cite{Stanley_2023}.  

A {\em partition} $\lambda=(\lambda_1,\lambda_2,\dots)$
 is a weakly decreasing sequence of nonnegative 
integers containing only finitely many non-zero terms.  
The {\em length} of a partition $\lambda$, which
is the number of positive parts, is denoted by $\ell(\lambda)$. {The sum of its parts is $\vert \lambda \vert =\sum_{i} \lambda_i$.}
A partition $\lambda$ can be represented pictorially as a {\em Young diagram}, {in which the} $i$th row contains $\lambda_i$ left-justified boxes. {The conjugate of $\lambda$ is $\lambda^\prime=(\lambda_1^\prime,\ldots,\lambda_s^\prime)$, where $\lambda_i^\prime$ is the number of boxes in the $i$th column of the Young diagram of shape $\lambda$.}

{For partitions $\lambda$ and $\mu$, where $\mu \subset \lambda$ \, (i.e. $\mu_i \leq \lambda_i$, for all $i \geq 1)$, the {\emph{skew
shape} $\lambda/\mu$} is the set theoretic difference where
$\lambda$ and $\mu$ are regarded as Young diagrams.} {Define $\vert\lambda/\mu\vert=\vert \lambda \vert - \vert \mu \vert$.}
For example, \cref{fig:lb} shows the Young diagram of $\lambda/\mu=(5,2,2)/(2,1)$.
\begin{figure}[H]
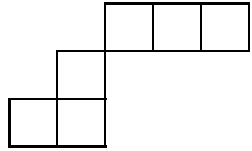

         \centering
\ydiagram{2+3,1+1,0+2}
\vspace{0.1in}
    \caption{The Young diagram of $\lambda/\mu=(5,2,2)/(2,1)$}
    \label{fig:lb}
\end{figure}
{A skew diagram $\lambda/\mu$ is a {\em horizontal $m$-strip} (respectively a {\em vertical $m$-strip}) if $\vert \lambda/\mu \vert =m$ and
$\lambda/\mu$ has at most one box in each column (respectively row).
Two boxes in a skew diagram are {\em adjacent} if they share a common side, and a subset of boxes in a skew diagram is {\em connected} if every pair of boxes is connected through a sequence of adjacent boxes. A {\em border strip} is a skew diagram that  is connected and contains no $2 \times 2$ block of boxes. The {\em height} of a border strip $\lambda/\mu$, denoted ht($\lambda/\mu$), is one less than the number of non-empty rows in $\lambda/\mu$.} We will write $\lambda/\mu \in \BS(r)$ (resp. $\HST(r)$ and $\VS(r)$) to indicate that $\lambda/\mu$ is a border strip (resp. horizontal strip and vertical strip) with $r$ boxes. We write $\lambda/\mu \in \BS$ (resp. $\HST$ and $\VS$) to indicate that $\lambda/\mu$ is a border strip (resp. horizontal strip and vertical strip).

Let $X = (x_1,\dots,x_n)$ be an $n$-tuple of commuting indeterminates. Consider the ring $\mathbb{Z}[X]$ of polynomials in $n$ commuting variables $x_1,\dots,x_n$ with integer coefficients. A polynomial in $\mathbb{Z}[X]$ is {\em symmetric} if it is invariant under the action of permuting the variables. Let $\Lambda$ denote the subring of symmetric polynomials. 

For a partition $\lambda$ of length at most $n$, the polynomial
\[
m_{\lambda}(X) \coloneqq \sum_{\alpha} x_1^{\alpha_1} x_2^{\alpha_2} \cdots x_n^{\alpha_n}, 
\]
where the sum is over all distinct permutations $\alpha=(\alpha_1,\dots,\alpha_n)$ of $\lambda$, is the {\em monomial symmetric function}. 

The {\em elementary symmetric function} $e_{\lambda}(X)$ indexed by a partition $\lambda$ is defined as
\begin{equation}
\label{def:ele}
e_{\lambda}(X) \coloneqq \prod_{i=1}^n e_{\lambda_i}(X),   
\end{equation}
where
\[
 e_r(X) \coloneqq \ds \sum_{1 \leq i_1 < i_2 < \dots < i_r \leq n} x_{i_1} x_{i_2} \dots x_{i_r} 
\text { for } r \geq 1 
\text{ and } e_0(X) \coloneqq 1.
\]
The {\em complete symmetric function} $h_{\lambda}(X)$ indexed by a partition $\lambda$ is defined as
\begin{equation}
\label{def:cmp}
h_{\lambda}(X) \coloneqq \prod_{i=1}^n h_{\lambda_i}(X),    
\end{equation}
where 
\[ h_r(X) \coloneqq  \ds \sum_{1 \leq i_1 \leq i_2 \leq \dots \leq i_r \leq n} x_{i_1} x_{i_2} \dots x_{i_r} 
 \text { for } r \geq 1  
 \text{ and } h_0(X) \coloneqq 1.
\]
{The sets $\{m_\lambda \mid \ell(\lambda) \leq n\}$,  $\{e_{\lambda}(X) \mid \lambda_1 \leq n\}$ and $\{h_{\lambda}(X) \mid \lambda_1 \leq n\}$ form $\mathbb{Z}$-bases for $\Lambda$.} 

The {\em power sum symmetric function} $p_{\lambda}(X)$ indexed by a partition $\lambda$ is defined as
\begin{equation}
\label{def:pwr}
 p_{\lambda}(X) \coloneqq  \prod_{i=1}^n p_{\lambda_i}(X),   
 \end{equation}
where
$ p_r(X) \coloneqq \sum_{i=1}^r x_i^r
    \text{ for } r \geq 1
$ and $p_0(X) \coloneqq 1$. The set $\{p_{\lambda}(X): \lambda_1 \leq n\}$ forms a $\mathbb{Q}$-basis for $\Lambda$.
It is convenient to define $e_r(X)$, $h_r(X)$ and $p_r(X)$ to be zero for $r<0$. 

{Another} $\mathbb{Z}$-basis for the ring of symmetric functions $\Lambda$ is the set of Schur polynomials $\{s_{\lambda}(X) \mid \ell(\lambda) \leq n\}$.
The {\em Schur polynomial} indexed by a partition $\lambda$ is defined as 
\[
s_{\lambda}(X) \coloneqq 
\frac{\ds \det_{1\le i,j\le n}\Bigl(x_i^{\lambda_j+n-j}\Bigr)}
{\ds \det_{1\le i,j\le n}\Bigl(x_i^{n-j}\Bigr)}.
\]
When the context is clear, we will often simply write $s_\lambda$ (resp. $m_\lambda, e_\lambda, h_\lambda, p_\lambda$) instead of $s_\lambda(X)$ (resp. $m_\lambda(X), e_\lambda(X), h_\lambda(X), p_\lambda(X)$).

Define a scalar product on $\Lambda$ by
\[
\langle h_{\lambda}, m_{\mu} \rangle = \delta_{\lambda \mu},
\]
for all partitions $\lambda$ and $\mu$. Then by~\cite[Equation 4.8]{macdonald-2015}, $\langle s_{\lambda}, s_{\mu} \rangle = \delta_{\lambda \mu},$ so that $\{s_{\lambda} \mid \ell(\lambda) \leq n\}$ forms an orthonormal basis of $\Lambda$.

{The following theorem is known as Pieri's rule.  It expresses the product of a Schur polynomial and a complete symmetric function as a linear combination of Schur polynomials}.
\begin{thm}[{\cite[Equation 5.16]{macdonald-2015}}]
    Let $\mu$ be a partition of length at most $n$ and $r \in \mathbb{Z}_{\geq 0}$. Then
    \[
    h_r s_{\mu}=\sum_{\gamma/\mu \, \in\, \HST(r)} s_{\gamma}.
    \]
\end{thm}
\noindent Another version is as follows.
\begin{thm}[{\cite[Equation 5.17]{macdonald-2015}}]
\label{thm:pieri}
    Let $\mu$ be a partition of length at most $n$ and $r \in \mathbb{Z}_{\geq 0}$. Then
    \[
    e_r s_{\mu}=\sum_{\gamma/\mu \, \in\, \VS(r)}  s_{\gamma}.
    \]
\end{thm}
The  Murnaghan--Nakayama rule is the following theorem, which 
expresses the product of a Schur polynomial and a power-sum symmetric function as a linear combination of Schur polynomials.
 \begin{thm}[{\cite[Theorem 7.17.3]{Stanley_2023}}]
\label{schurmn}
Let $\mu$ be a partition of length at most $n$  and $r \in \mathbb{Z}_{\geq 0}$. We have 
\[
p_r s_{\mu}=\sum_{\gamma/\mu \, \in\, \BS(r)} (-1)^{\hgt(\gamma/\mu)} s_{\gamma}.
\]
 \end{thm}

\begin{eg} We have 
$$p_4 s_{(3,1)}=-s_{(3,1,1,1,1,1)}+s_{(3,2,2,1)}-s_{(3,3,2)}-s_{(4,4)}+s_{(7,1)}.$$ The partitions in the expansion, with border strips highlighted, are displayed in \cref{fig:MNeg}.\end{eg}

\begin{figure}[H]
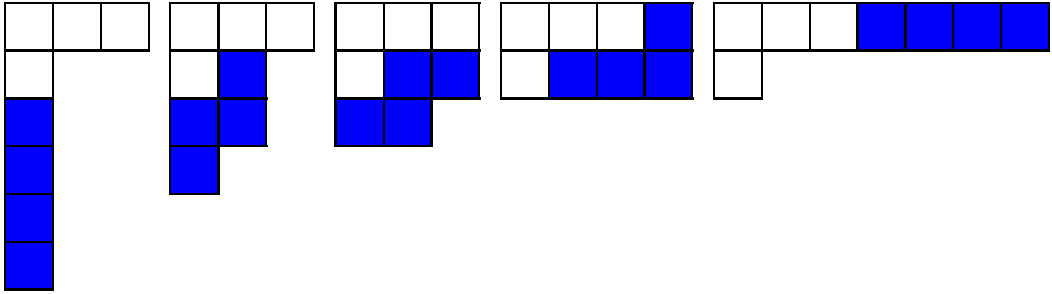

         \centering

\ydiagram[*(blue)]
{0,0,1,1,1,1}
*[*(white)]{3,1,1,1,1,1}\ \
\ydiagram[*(blue)]
{0,1+1,2,1}
*[*(white)]{3,2,2,1}\ \
\ydiagram[*(blue)]
{0,1+2,2}
*[*(white)]{3,3,2}\ \
\ydiagram[*(blue)]
{3+1,1+3}
*[*(white)]{4,4}\ \ 
\ydiagram[*(blue)]
{3+4,1+0}
*[*(white)]{7,1} 
\vspace{0.1in}
    \caption{The partitions in the expansion $p_4s_{(3,1)}$}
    \label{fig:MNeg}
\end{figure}


{Skew Schur polynomials $s_{\lambda/\mu}(X)$   are 
generalizations of Schur polynomials.} The combinatorial objects associated with $s_{\lambda/\mu}(X)$ are semistandard tableaux of shape $\lambda/\mu$. 
A {\em semistandard tableau} of shape $\lambda/\mu$ is a filling of the Young diagram  $\lambda/\mu$ with entries in $\{1, 2, \dots,n \}$ such that entries increase weakly across rows and strictly down columns. 

Define the
{\em skew Schur polynomial} $s_{\lambda/\mu}(X)$ as 
\begin{equation}
\label{def:tab-skew-schur}
s_{\lambda/\mu}(X)=\sum_{T} \left(\prod_{i=1}^n  x_i^{n_i(T)} \right),    
\end{equation}
where the sum is taken over all semistandard tableaux of shape $\lambda/\mu$ and $n_i(T)$, $i \in [n]$ is the number of occurrences of $i$ in $T$. {When $\mu=\emptyset$, $s_{\lambda/\mu}$ is an ordinary Schur polynomial.}  Since the entries increase strictly down columns in a tableau,  
\begin{equation}
\label{ss-rec}
    s_{\lambda/\mu}(X)= \sum_{\substack{\mu \,\subset \, \nu \, \subset \, \lambda \\
    \lambda/\nu \, \in \,\HST}} s_{\nu/\mu}(x_1,\dots,x_{n-1}) 
    x_n^{|\lambda|-|\nu|}.
\end{equation}
For each symmetric function $f \in \Lambda$, {define $f^{\perp}: \Lambda \rightarrow \Lambda$ to be the adjoint of multiplication by $f$},
i.e. $\langle f^{\perp} u, v \rangle=\langle u, fv \rangle$ for all $u,v \in \Lambda$ (see~\cite[Example I.5.3]{macdonald-2015}).
Suppose $\lambda$ is a partition of length at most $n$. Then 
 
 \[
 \langle  p_r^{\perp} s_{\lambda}, s_{\mu} \rangle =  
 \langle  s_{\lambda}, p_r s_{\mu} \rangle = \Big \langle   s_{\lambda}, \sum_{\eta/\mu \, \in \, \BS(r)} (-1)^{\hgt(\eta/\mu)} s_{\eta} \Big \rangle = \begin{cases}
     (-1)^{\hgt(\lambda/\mu)}  & \mbox{ if } \lambda/\mu \in \BS(r) \\
     0 & \mbox{ if } \lambda/\mu \not \in \BS(r),
 \end{cases}
 \]
where the second equality uses \cref{schurmn}. It follows that
 \begin{equation}
 \label{pperp}
 p_r^{\perp} s_{\lambda} 
 = \sum_{\xi} \langle  p_r^{\perp} s_{\lambda}, s_{\xi} \rangle s_{\xi} = 
 \sum_{\lambda/\xi \, \in \, \BS(r)}  (-1)^{\hgt(\lambda/\xi)} s_{\xi}.
 \end{equation}
Using \cref{thm:pieri}, we have
  \[
 \langle  e_r^{\perp} s_{\lambda}, s_{\mu} \rangle =  
 \langle  s_{\lambda}, e_r s_{\mu} \rangle = \Big \langle   s_{\lambda}, \sum_{\zeta/\mu \, \in \, \VS(r)}  s_{\zeta} \Big \rangle = \begin{cases}
     1 & \mbox{ if }\lambda/\mu \in \VS(r) \\
     0 & \mbox{ if }\lambda/\mu \not \in \VS(r).
     \end{cases}
 \]
Therefore,
  \begin{equation}
  \label{eperp}
 e_r^{\perp} s_{\lambda} = \sum_{\pi} \langle  e_r^{\perp} s_{\lambda}, s_{\pi} \rangle s_{\pi} = 
 \sum_{\lambda/\pi \, \in \, \VS(r)}  s_{\pi}.
 \end{equation}

The characters of irreducible polynomial representations of the symplectic and orthogonal groups  are symmetric Laurent polynomials indexed by partitions. They are given by the following Weyl character formulas. Define $\x = 1/x$ for an indeterminate $x$.
 The \emph{odd orthogonal character} 
 indexed by $\lambda = (\lambda_1,\dots,\lambda_n)$ is given by
\begin{equation}
\label{oodef}
\oo_\lambda(X) \coloneqq
\frac{\ds \det_{1\le i,j\le n}\Bigl(x_i^{\lambda_j+n-j+1/2}-\x_i^{\lambda_j+n-j+1/2}\Bigr)}
{\ds \det_{1\le i,j\le n}\Bigl(x_i^{n-j+1/2}-\x_i^{n-j+1/2}\Bigr)}.
\end{equation}
The symplectic character indexed by $\lambda = (\lambda_1,\dots,\lambda_n)$ is given by  
\begin{equation}
\label{symp:def}
   \sp_{\lambda}(X) \coloneqq
\frac{\ds \det_{1\le i,j\le n}\Bigl(x_i^{\lambda_j+n-j+1}-\x_i^{\lambda_j+n-j+1}\Bigr)}
{\ds \det_{1\le i,j\le n}\Bigl(x_i^{n-j+1}-\x_i^{n-j+1}\Bigr)}. 
\end{equation}
Lastly, the \emph{even orthogonal character} 
indexed by $\lambda = (\lambda_1,\dots,\lambda_n)$ is given by
\begin{equation}
\label{oedef}
\oe_\lambda(X) \coloneqq
\frac{2 \displaystyle\det_{1\le i,j\le n}\Bigl(x_i^{\lambda_j+n-j}+\x_i^{\lambda_j+n-j}\Bigr)}
{\displaystyle (1+\delta_{\lambda_n,0})
\det_{1\le i,j\le n}\Bigl(x_i^{n-j}+\bar{x}_i^{n-j}\Bigr)},
\end{equation}
where $\delta$ is the Kronecker delta. 

{The above characters are also referred to as orthogonal Schur polynomials and symplectic Schur polynomials.  They can be described in terms of symplectic and orthogonal tableaux (see \cite{sundaram-symplectic} and \cite{sundaram-orthogonal}).  In the symplectic case, the semistandard symplectic (King) tableaux have entries from $1<\ov{1}<2<\ov{2}<\cdots<n<\ov{n}$.  They are semistandard with respect to this ordering and satisfy the additional condition that every entry in the $i$th row is greater than or equal to $\ov{i}$. We have
$$\ds \sp_\lambda(X)=\sum_T\left (\prod_{i=1}^n x_i^{n_i(T)-n_{\ov{i}}(T)}\right),$$ where the sum runs over all semistandard symplectic $\lambda$-tableaux.}  

We will write $\sp_\lambda$ (resp. $\oo_\lambda$ and $\oe_\lambda$) instead of $\sp_\lambda(X)$ (resp. $\oo_\lambda(X)$ and $\oe_\lambda(X)$) whenever the context is clear.

{\begin{eg} The five semistandard symplectic tableaux of shape $\lambda=(1,1)$ and the corresponding symplectic Schur function are given below.  
$$\begin{ytableau} 1 \cr 2\cr  \end{ytableau} \ \  \begin{ytableau}1 \cr \ov{2}\cr \end{ytableau} \ \ 
\begin{ytableau}\ov{1} \cr 2 \cr \end{ytableau} \ \  \begin{ytableau}\ov{1} \cr \ov{2}\cr\end{ytableau} \ \ 
\begin{ytableau}2 \cr \ov{2} \end{ytableau} \ ; \ \ \ \sp_{(1,1)}(x_1,x_2)=x_1x_2+x_1x_2^{-1}+x_1^{-1}x_2+x_1^{-1}x_2^{-1}+1$$
    \end{eg}} 

\bigskip

We now give a supersymmetric analogue of the ring of  symmetric functions. Suppose $Y=(y_1,\dots,y_m)$. Consider the ring $\mathbb{Z}[X,Y]$ of polynomials in $(n+m)$ commuting variables $x_1,\dots,x_n,$ $y_1,\dots,y_m$ with integer coefficients.
A polynomial $f(X,Y)$ in this ring is {\em doubly symmetric} if it is symmetric in both $(x_1,\dots,x_n)$ and $(y_1,\dots,y_m)$. Moreover, if substituting $x_n=t$ and $y_m=-t$ results in an expression independent of $t$, then we call $f(X,Y)$ a {\em supersymmetric function}. See~\cite{moensthesis} for more details and background. 
The {\em elementary supersymmetric function} indexed by $\lambda=(\lambda_1,\lambda_2,\dots,\lambda_n)$ is given by 
\[
 E_{\lambda}(X/Y) \coloneqq
 \prod_{i=1}^n E_{\lambda_i}(X/Y), 
\]
where
\[
E_r(X/Y) \coloneqq  \sum_{j=0}^{r} 
 e_{j}(X) h_{r-j}(Y), \, r \geq 1 \text{ and } E_0(X/Y) \coloneqq 1.
\]
The {\em complete supersymmetric function} indexed by $\lambda=(\lambda_1,\lambda_2,\dots,\lambda_n)$ is given by 
\[
 H_{\lambda}(X/Y) \coloneqq 
 \prod_{i=1}^n H_{\lambda_i}(X/Y),
\]
where
\[
 H_r(X/Y) \coloneqq  \sum_{j=0}^{r} 
 h_{j}(X) e_{r-j}(Y), \, r \geq 1 \text{ and } H_0(X/Y) \coloneqq 1.
\]
The {\em power sum supersymmetric function} indexed by $\lambda=(\lambda_1,\lambda_2,\dots,\lambda_n)$ is given by
\begin{equation}
\label{def:pwr-sup}
 P_{\lambda}(X/Y) \coloneqq  \prod_{i=1}^n P_{\lambda_i}(X/Y),   
 \end{equation}
where
\[P_r(X/Y) \coloneqq \sum_{i=1}^r (x_i^r+(-1)^{r-1}y_i^r), \, r \geq 1 \text{ and } P_0(X/Y) \coloneqq 1.\]
The {\em hook Schur} or {\em supersymmetric Schur polynomial} indexed by a partition $\lambda$ is given by  
\[
\hs_{\lambda}(X/Y) \coloneqq \sum_{\mu} \s_{\mu}(X) \s_{\lambda'/\mu'}(Y).
\]
Hook Schur polynomials can be described in terms of tableaux that have both a row-strict and a column-strict part.

Note that the hook Schur polynomial $\hs_{\lambda}(X/Y)$ is nonzero if and only if $\lambda_{n+1} \leq m$.  Stembridge~\cite{stembridge1985characterization} showed that the set $\{\hs_{\lambda}(X/Y) \mid 
\lambda_{n+1} \leq m\}$ forms a basis for the ring of supersymmetric functions. Using \eqref{ss-rec}, we have
\begin{equation}
    \hs_{\lambda}(X/Y)=\sum_{\lambda/\pi \, \in \, \VS}  \hs_{\pi}(X/\underline{Y}) y_m^{|\lambda|-|\pi|}.
\end{equation}
We can also write 
supersymmetric functions using the notion of plethystic difference of two sets of variables. 
See~\cite{loehr2011computational} for a nice exposition of combinatorial rules underlying plethystic calculus. We have
\[
P_{\lambda}(X/Y)=p_{\lambda}[X-\epsilon Y]|_{\epsilon=-1} \text{ and }
\hs_{\lambda}(X/Y)=s_{\lambda}[X-\epsilon Y]|_{\epsilon=-1}.\]
As a consequence, we get the following supersymmetric version of the Murnaghan--Nakayama rule given in \cref{schurmn}.
\begin{thm}
\label{thm:hook} 
Suppose $\lambda$ is a partition of length at most $n$ and $r \geq 0$.
We have
\[
P_r(X/Y)\hs_{\lambda}(X/Y)=\sum_{\eta/\lambda \, \in\, \BS(r)} (-1)^{\hgt(\eta/\lambda)} \hs_{\eta}(X/Y).
\]
 \end{thm}
{Orthosymplectic Schur polynomials are hybrids of symplectic Schur polynomials and ordinary Schur polynomials.  }
The {\em orthosymplectic Schur polynomial} indexed by a partition $\lambda$ is given by  
\[
\spo_{\lambda}(X/Y) \coloneqq \sum_{{\mu \subseteq \lambda}} \sp_{\mu}(X) \s_{\lambda'/\mu'}(Y).
\]

Note that  
$\spo_{\lambda}(X;\emptyset)=\sp_{\lambda}(X)$ and, as with hook Schur polynomials,  $\spo_{\lambda}(X/Y)$ is  nonzero if and only if $\lambda_{n+1} \leq m$.  Moreover,
\begin{equation}
\label{sporec}
    \spo_{\lambda}(X/Y)=\sum_{\lambda/\pi \, \in \, \VS}  \spo_{\pi}(X/\underline{Y}) y_m^{|\lambda|-|\pi|}.
\end{equation}

{ Orthosymplectic Schur polynomials are characters of orthosymplectic Lie superalgebras; for further background, see \cite{benkart-1988}.  
A tableaux description of orthosymplectic Schur polynomials is given in \cite{benkart-1988}.  An $\spo(2n,m)$-tableau $T$ has entries in $1<\ov{1}<\cdots < n < \ov{n}<1^\prime < \cdots <m^\prime$.  The portion that contains entries from $\{1, \ov{1},\ldots,n,\ov{n}\}$ is semistandard symplectic and the remaining skew tableau is strictly increasing across rows and weakly increasing down columns. Then $$\ds \spo_\lambda(X/Y)=\sum_T\left (\prod_{\substack{1 \leq i \leq n \\ 1 \leq j \leq m}} x_i^{n_i(T)-n_{\ov{i}}(T)}y_j^{n_{j^\prime}(T)}\right),$$ where the sum runs over all $\spo(2n,m)$-tableaux of shape $\lambda$.}

{\begin{eg}  
(a) The following is an $\spo(4,2)$-tableau of shape $\lambda=(3,2)$: 
\begin{ytableau}2 & 1^\prime & 2^\prime \cr
\ov{2} & 1^\prime \cr \end{ytableau}\\
(b) Let $n=2$ and $m=1$.  Then 
$$\spo_{(1,1)}(X/Y)=
x_1x_2+x_1x_2^{-1}+x_1^{-1}x_2+x_1^{-1}x_2^{-1}+1+x_1y_1+x_1^{-1}y_1+x_2y_1+x_2^{-1}y_1+y_1^2.$$
\end{eg}}

\section{Symplectic and orthogonal Murnaghan--Nakayama rules}\label{sec:s-o-rules}
In this section, we prove Murnaghan--Nakayama {rules} for characters of symplectic and orthogonal groups.
 For $\alpha=(\alpha_1,\dots,\alpha_n) \in \mathbb{N}^n \cup (\mathbb{N}+\frac{1}{2})^n$, let \[
A_{\alpha}
=\det_{1\le i,j\le n}
\Bigl(x_i^{\alpha_j}-\x_i^{\alpha_j}\Bigr).
\]

\noindent {By \eqref{symp:def}, $\displaystyle \mbox{sp}_\lambda(X)=\frac{A_{\lambda+\delta}}{A_\delta},$ where $\delta=(n,n-1,\ldots,1)$.} Write $\X = (\x_1,\dots,\x_n)$ and define $$\overline{p}_r \coloneqq p_r(X,\X)={\sum_{i=1}^n} (x_i^r+\x_i^r).$$
\begin{lem}
\label{lem:MN}

Suppose $\epsilon_j=(0,\dots,0,1,0,\dots,0)$.
We have 
    \begin{equation}
    \label{lem}
         \overline{p}_r A_{\alpha}  = 
     \sum_{j=1}^n A_{\alpha+r \epsilon_j}+\sum_{j=1}^n A_{\alpha-r \epsilon_j}.
    \end{equation}
\end{lem}
\begin{proof}
 We proceed by induction on $n$. If $n=1$, then 
 \[
  p_r(x_1,\x_1) \Bigl(x_1^{\alpha_1+1}-\x_i^{\alpha_1+1}\Bigr) = x_1^{\alpha_1+r+1}-\x_i^{\alpha_1+r+1} + x_1^{\alpha_1-r+1}-\x_i^{\alpha_1-r+1},
 \]
and \eqref{lem} holds. 
 Assume \cref{lem:MN} holds if the order of $A_{\alpha}$ is strictly less than $n$.
 Let $\phi_r=x_n^{r}-\x_n^r$, $\Delta_j=(\underbrace{1,\dots,1}_j,0,\dots,0)$ and $\alpha^j=(\alpha_1,\dots,\widehat{\alpha}_j,\dots,\alpha_n)$. Notice that 
  \[
A_{\alpha}=(-1)^{n+1}\phi_{\alpha_1+n} A_{\alpha^1}+(-1)^{n+2}\phi_{\alpha_2+n-1} A_{\alpha^2+\Delta_1}+
     \dots+
     (-1)^{2n}\phi_{\alpha_n+1} A_{\alpha^n+\Delta_{n-1}}.
     \]
Let $B_{\alpha}=\sum_{j=1}^{n} (A_{\alpha+r \epsilon_j}+A_{\alpha-r \epsilon_j})$ for $\alpha=(\alpha_1,\dots,\alpha_n)$ and $\psi_{s+r}=\phi_{s+r}+\phi_{s-r}$. We have 
\begin{equation*}
    \begin{split}
         p_r(X,\X) A_{\alpha} =
         \Bigg( (-1)^{n+1}\phi_{\alpha_1+n}   
     B_{\alpha^1} &+  (-1)^{n+2}\phi_{\alpha_2+n-1}
     B_{\alpha^2+\Delta_1}+
     \dots \\
     +(-1)^{2n}\phi_{\alpha_n+1} 
      & B_{\alpha^n+\Delta_{n-1}} + (-1)^{n+1} \psi_{\alpha_1+n+r} A_{\alpha^1}
     \\
     +(-1)^{n+2}& \psi_{\alpha_2+n-1+r} A_{\alpha^2+\Delta_1}+
     \dots+
     (-1)^{2n}\psi_{\alpha_n+1+r} A_{\alpha^n+\Delta_{n-1}} \Bigg).
    \end{split}    
\end{equation*}
Since
    \begin{equation*}
        \begin{split}
         A_{\alpha+r \epsilon_j} = & 
     (-1)^{n+1}\phi_{\alpha_1+n} A_{\alpha^1+r\epsilon_{j-1}} 
     +\dots + 
     (-1)^{n+j-1} \phi_{\alpha_{j-1}+n-j+2} A_{\alpha^{j-1}+\Delta_{j-2}+r \epsilon_{j-1}}\\
    & \qquad \qquad \qquad + (-1)^{n+j} \phi_{\alpha_j+n+1-j+r} A_{\alpha^j+\Delta_{j-1}} +(-1)^{n+j+1} \phi_{\alpha_{j+1}+n-j} A_{\alpha^{j+1}+\Delta_j+r \epsilon_j}
    \\
    & \qquad \qquad \qquad \qquad \qquad   + \dots 
     + (-1)^{2n} \phi_{\alpha_n+1} A_{\alpha^n+\Delta_{n-1}+r\epsilon_j},
      \end{split}
    \end{equation*}
\eqref{lem} follows.  
\end{proof}

{We now define some key terms that will be used throughout the remainder of the paper.}
Given a partition $\mu$ of length at most $n$ and $\delta=(n,n-1,\ldots,1)$, define
\begin{equation}
\label{m}
 m(\mu)=
\begin{cases}
0 & \mbox{ if } \mu_i+\delta_i <r \mbox{ for } i \in [1,n]\\
\mbox{max}\{ i \mid \mu_i +\delta_i \geq r\} & \mbox{otherwise}.\\
\end{cases}   
\end{equation}


\begin{defn}
\label{def:muq}
    Let $\mu$ be a partition of length at most $n$, $\delta=(n,n-1,\dots,1)$, $r \geq 1$  an integer 
    and $q \in [1,n]$. If $\mu_q+\delta_q-r < 0$, {then arrange 
    \[
    (\mu_1+\delta_1,\dots,\mu_{q-1}+\delta_{q-1},r-(\mu_q+\delta_q),\mu_{q+1}+\delta_{q+1},\dots,\mu_n+\delta_n)
    \]}in descending order to form $\Delta_q$. Define $\mu^{(q)}=\Delta_q-\delta$. 
If $r-(\mu_q+\delta_q)=\mu_j+\delta_j$ for some $j \neq q$, then $\mu^{(q)}$ will not be a partition and we take $\sp_{\mu^{(q)}}=0$. Otherwise,   $\mu^{(q)}$ will be a partition given as follows. Suppose $p(q)$ is the index of $r-(\mu_q+\delta_q)$ in $\Delta_q$. If $p(q) \geq q$,
    \[
   \mu^{(q)}=(\mu_1,\dots,\mu_{q-1},\mu_{q+1}-1,\dots,\mu_{p(q)}-1,r-(\mu_q+\delta_q+\delta_{p(q)}),
   \mu_{p(q)+1},\dots,\mu_n).
    \]
Otherwise, i.e., if $p(q) < q$, then
    \[
   \mu^{(q)} = (\mu_1,\dots,\mu_{p(q)-1}, r-(\mu_q+\delta_q+\delta_{p(q)}),
   \mu_{p(q)}+1,\dots,\mu_{q-1}+1,\mu_{q+1},\dots,\mu_n).
    \]    
\end{defn}

\bigskip

{\begin{eg} Let
    $\mu=(3,2,1,0)$, $n=4$, $r=9$ and $q=3$. Then $\mu+\delta=(7,5,3,1)$.  Rearrange $(7,5,6,1)$ to get $\Delta_q=(7,6,5,1)$. Then $p(q)=2$ and $\mu^{(q)}=(3,3,3,0)$.
\end{eg}}
\bigskip

\noindent {The following is a symplectic version of the Murnaghan--Nakayama rule.}
\begin{thm} 
\label{thm:smp}
 Let $\mu$ be a partition of length at most $n$
 and $r \geq 1$  an integer.
 Then
 \begin{equation}
 \label{sp-MN}
     \overline{p}_r \sp_{\mu} = \sum_{ \eta/\mu \,\in \, \BS(r)} (-1)^{\hgt(\eta/\mu)}  \sp_{\eta}
        +\sum_{\mu/\xi \, \in \, \BS(r)} (-1)^{\hgt(\mu/\xi)} \sp_{\xi}
       + \sum_{q=m(\mu)+1}^n (-1)^{p(q)-q+1} \sp_{\mu^{(q)}}.
 \end{equation}
\end{thm}
 \begin{proof}
 Suppose $\delta=(n,n-1,\dots,1)$. {We have}
  \begin{equation}
  \label{1}
   p_r(X,\X) \sp_{\mu}(X)  = 
     \sum_{q=1}^n \frac{A_{\mu+\delta+r \epsilon_q}}{A_{\delta}}
     +\sum_{q=1}^n \frac{A_{\mu+\delta-r \epsilon_q}}{A_{\delta}}.
  \end{equation}
 Fix $q \in [1,n]$ and arrange $\mu+\delta+r \epsilon_q$ in descending order to form $\Theta$. If $\Theta$ has two equal coordinates, {then $A_{\mu+\delta+r \epsilon_q}=0$.} 
Otherwise, let $p(q)$ be the index of $\mu_q+\delta_q+r$ in $\Theta$. Note that $p(q) \leq q$
and $A_{\mu+\delta+r \epsilon_j}=(-1)^{p(q)-q} A_{\eta+\delta}$, where    
 \[
 \eta=(\mu_1,\dots,\mu_{p(q)-1},\mu_q+p(q)-q+r,\mu_{p(q)}+1,\dots,\mu_{q-1}+1,\mu_{q+1},\dots,\mu_n).
 \]
 Notice that $\eta/\mu$ is a border strip of height $q-p(q)$.
 Therefore,
 \begin{equation}
 \sum_{q=1}^n \frac{A_{\mu+\delta+r \epsilon_q}}{A_{\delta}}
 =\sum_{\eta/\mu \,\in \, \BS(r)} (-1)^{\hgt(\eta/\mu)} \sp_{\eta}(X).
 \end{equation}
Now we consider the second sum in the right-hand side of \eqref{1}. Fix $q \in [1,m(\mu)]$. Then $\mu_q+\delta_q-r \geq 0$. If $\mu_q+\delta_q-r = 0$, then $A_{\mu+\delta-r \epsilon_q} =0.$
If $\mu_q+\delta_q-r > 0$,  arrange 
$\mu+\delta-r \epsilon_q$ in descending order to get $\Pi$. Notice that {$A_{\mu+\delta-r \epsilon_q}=0$ 
if two coordinates of $\Pi$ are equal, so assume all coordinates of $\Pi$ are distinct and that $p(q)$ is the index of $\mu_q+\delta_q-r$ in $\Pi$.} 
Note that $p(q) \geq q$ and $A_{\mu+\delta-r \epsilon_q}=(-1)^{p(q)-q} A_{\xi+\delta}$, where    
 \[
 \xi=(\mu_1,\dots,\mu_{q-1},\mu_{q+1}-1,\dots,\mu_p(q)-1,\mu_q+p(q)-q-r,\mu_{p(q)+1},\dots,\mu_n)
 \]
and $\mu/\xi$ is a border strip of height $p(q)-q$.  
Therefore,
 \begin{equation}
 \sum_{q=1}^{m(\mu)} \frac{A_{\mu+\delta-r \epsilon_q}}{A_{\delta}}=\sum_{\mu/\xi \,\in \, \BS(r)} (-1)^{\hgt(\mu/\xi)} \sp_{\xi}(X).
 \end{equation}
Lastly fix $q \in [m(\mu)+1,n]$, so that $\mu_q+\delta_q-r\epsilon_q <0$.
If $r-(\mu_q+\delta_q)=\mu_j+\delta_j$ for some $j \neq q$, then $A_{\mu+\delta-r \epsilon_q} =0.$ Otherwise, arrange
\[
    (\mu_1+\delta_1,\dots,\mu_{q-1}+\delta_{q-1},r-\mu_q+\delta_q,\mu_{q+1}+\delta_{q+1},\dots,\mu_n+\delta_n)
    \]
in descending order to obtain $\Delta_q$ and
suppose $p(q)$ is the index of $r-\mu_q+\delta_q$ in $\Delta_q$. Thus
\begin{equation}
 \sum_{q=m(\mu)+1}^{n} \frac{A_{\mu+\delta-r \epsilon_q}}{A_{\delta}}=
 \sum_{q=m(\mu)+1}^n (-1)^{p(q)-q+1} \sp_{\mu^{(q)}},
 \end{equation}
where, if $p(q) \geq q$, we have
    \[
   \mu^{(q)}=(\mu_1,\dots,\mu_{q-1},\mu_{q+1}-1,\dots,\mu_{p(q)}-1,r-(\mu_q+\delta_q+\delta_{p(q)}),
   \mu_{p(q)+1},\dots,\mu_n)
    \]
and, if $p(q) < q$, we have
    \[
   \mu^{(q)}  =(\mu_1,\dots,\mu_{p(q)-1}, r-(\mu_q+\delta_q+\delta_{p(q)}),
   \mu_{p(q)}+1,\dots,\mu_{q-1}+1,\mu_{q+1},\dots,\mu_n).
    \] This completes the proof.
 \end{proof}

 {When $\mu_{n}+1 \geq r$, \cref{sp-MN} can be written more concisely.}
 \begin{cor} 
 Let $\mu$ be a partition of length $n$ such that $\mu_n+1 \geq r$. {Then}  
 \begin{equation}
     \overline{p}_r \sp_{\mu}
       =\sum_{\eta/\mu \, \in \, \BS(r)} (-1)^{\hgt(\eta/\mu)} \sp_{\eta}
       +\sum_{\mu/\xi \, \in \, \BS(r)} (-1)^{\hgt(\mu/\xi)} \sp_{\xi}.
 \end{equation}
 \end{cor}

Combinatorially, partitions in the first sum in \eqref{sp-MN}  are found by adding border strips of length $r$ to $\mu$, while partitions in the second sum arise from removing length $r$ border strips from $\mu$.  Non-zero partitions $\mu^{(q)}$ in the third sum can be found as follows. 
\begin{enumerate}\label{comb}
\item Remove row $q$ from $\mu$.
\item If $p(q)<q$, shift rows $p(q)$ to $q-1$ down one row and add a box to each.
\item If $p(q)>q$, shift rows $q+1$ to $p(q)$ up one row and remove a box from each.
\item Add a row with $r-(\mu_q+\delta_q+\delta_{p(q)})$ boxes in row $p(q)$ to get $\mu^{(q)}$.
\end{enumerate}


\begin{eg}\label{eg:comb}  Suppose $n=3$, $\mu=(4,3,1)$ and $r=6$ and consider the product $\overline{p}_r\sp_\mu$.    Then
 $\mu+\delta=(7,5,2)$ and $m(\mu)=1.$  The partitions that arise from adding length $6$ border strips correspond to the first three summands of \eqref{eg:spproduct} below and the fourth summand comes from removing a length $6$ border strip.  Partitions in the remainder of the expansion arise as follows, with boxes to be removed highlighted.

 {For $q=2$, rearrange $(7,1,2)$ to get $\Delta_q=(7,2,1)$ so $p(q)=3$:}
$$\ydiagram[*(red)]{0,3,0}*[*(white)]{4,3,1}\  \xrightarrow{\hspace{1cm}} \ \ydiagram[*(red)]{0,1}*[*(white)]{4,1}\ \xrightarrow{\hspace{1cm}} \ydiagram{4}$$

{For $q=3$, $\Delta_q=(7,5,4)$ and $p(q)=3$:}
$$\ydiagram[*(red)]{0,0,1}*[*(white)]{4,3,1}\   \xrightarrow{\hspace{1cm}} \ydiagram{4,3,3}
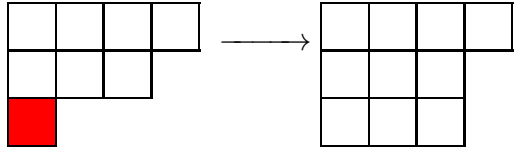
$$

\begin{equation}
    \label{eg:spproduct}
{\overline{p}_6
\mbox{sp}_{(4,3,1)}=\sp_{(5,5,4)}-\sp_{(8,5,1)}+\sp_{(10,3,1)}+\sp_{(2)}+\sp_{(4)}-\sp_{(4,3,3)}}
\end{equation}
 \end{eg}

\bigskip

 \cref{thm:oddortho} and \cref{thm:evenortho} are orthogonal Murnaghan--Nakayama rules.  \cref{thm:oddortho} is given by  working with $\delta=(n-1/2,n-3/2,\dots,1/2)$ instead of $\delta=(n,n-1,\dots,1)$ in the proof of \cref{thm:smp}. In the statement of the \cref{thm:oddortho} below, take $\delta = (n-1/2,n-3/2,\dots,1/2)$ in \eqref{m} and  \cref{def:muq} to get $m_o(\mu)$ and $\mu_o^{(q)}$.

\begin{thm} 
\label{thm:oddortho}
{Let $\mu$ be a partition of length at most $n$ and $r \geq 1$  an integer.
 Then
 \begin{equation*}
 \label{sp-oo}
     \overline{p}_r \oo_{\mu} = \sum_{ \eta/\mu \,\in \, \BS(r)} (-1)^{\hgt(\eta/\mu)}  \oo_{\eta}
        +\sum_{\mu/\xi \, \in \, \BS(r)} (-1)^{\hgt(\mu/\xi)} \oo_{\xi}
       + \sum_{q=m_o(\mu)+1}^n (-1)^{p(q)-q+1} \oo_{\mu_o^{(q)}}.
 \end{equation*}}
 \end{thm} 

\begin{eg}
Suppose $n=3$, $\mu=(2,1)$ and $r=2$ and consider the product $\overline{p}_r\oo_\mu$.   Then
 $\mu+\delta=(9/2,5/2,1/2)$ and $m_o(\mu)=2.$  The partition $(4,1)$ in \eqref{eg:ooproduct} arises from adding a border strip of length $2$. If $q=3$, then $r-(\mu_q+\delta_q)=3/2$, $p(q)=3$ and $\mu_o^{3}=(2,1,1)$. Therefore, 
    \begin{equation}
    \label{eg:ooproduct}
{\overline{p}_3
\oo_{(2,1)}=\oo_{(4,1)}-\oo_{(2,1,1)}.}
\end{equation}
\end{eg} 



\bigskip

For $\alpha=(\alpha_1,\dots,\alpha_n) \in \mathbb{N}^n$, let \[
N_{\alpha}
=\det_{1\le i,j\le n}
\Bigl(x_i^{\alpha_j}+\x_i^{\alpha_j}\Bigr).
\]

\noindent {By \eqref{oedef}, $$\displaystyle \oe_\lambda(X)=\frac{2 \, N_{\lambda+\delta}}{(1+\delta_{\lambda_n,0}) \, N_\delta},$$ where $\delta=(n-1,\ldots,1,0)$.} An approach similar to that in the proof of \cref{lem:MN} yields the following result. 
\begin{lem}
\label{lem:even}
Suppose $\epsilon_j=(0,\dots,0,1,0,\dots,0)$.
We have 
    \begin{equation}
         \overline{p}_r N_{\alpha}  = 
     \sum_{j=1}^n N_{\alpha+r \epsilon_j}+\sum_{j=1}^n N_{\alpha-r \epsilon_j}.
    \end{equation}
\end{lem}

\cref{thm:evenortho} is obtained by working with $N_{\alpha}$ (resp. $\delta=(n-1,n-2,\dots,0)$) instead of $A_{\alpha}$ (resp. $\delta=(n,n-1,\dots,1)$) in the proof \cref{thm:smp}. In the statement of  \cref{thm:evenortho} below, take $\delta = (n-1,n-2,\dots,0)$ in \eqref{m} and \cref{def:muq} to get $m_e(\mu)$ and $\mu_e^{(q)}$.

\begin{thm} 
\label{thm:evenortho}
{Let $\mu$ be a partition of length at most $n$ and $r \geq 1$  an integer.
 Then
 \begin{equation*}
 \label{sp-eoo}
 \begin{split}
     \overline{p}_r \oe_{\mu} 
     = \sum_{ \eta/\mu \,\in \, \BS(r)} \frac{(-1)^{\hgt(\eta/\mu)}(1+\delta_{\eta_n,0})}{1+\delta_{\mu_n,0}}  & \oe_{\eta} 
        +
        \sum_{\mu/\xi \, \in \, \BS(r)} 
        \frac{(-1)^{\hgt(\mu/\xi)}(1+\delta_{\xi_n,0})
        }{1+\delta_{\mu_n,0}} \oe_{\xi}
       \\  & 
       + \sum_{q=m_e(\mu)+1}^n 
       \frac{(-1)^{p(q)-q}(1+\delta_{(\mu_e^{(q)})_n,0})}{1+\delta_{\mu_n,0}} \oe_{\mu_e^{(q)}}.
       \end{split}
 \end{equation*}} 
 \end{thm}

\begin{eg}
 Suppose $n=3$, $\mu=(2,1)$ and $r=3$ and consider the product $\overline{p}_r\oe_\mu$.   Then
 $\mu+\delta=(4,2,0)$ and $m_e(\mu)=1.$  The partitions in the first four summands of \eqref{eg:eoproduct} arise from adding and removing  border strips of length $3$. If $q=2$, then $p(q)=2$ and   $\mu_o^{(q)}=(2)$. If $q=3$, then $p(q)=2$ and $\mu_o^{(q)}=(2,2,2)$. Therefore, 
     \begin{equation}
    \label{eg:eoproduct}
\overline{p}_3
\oe_{(2,1)}=\oe_{(5,1)}-\oe_{(3,3)}- \frac{1}{2} \, \oe_{(2,2,2)}-\oe_{\emptyset}+ 
\oe_{(2)}- \frac{1}{2} \, \oe_{(2,2,2)}.
\end{equation}
 \end{eg}
\section{An orthosymplectic Murnaghan--Nakayama rule}\label{sec:osprule}

In this section, we prove a Murnaghan--Nakyama rule for orthosymplectic characters. To that end, we first prove some lemmas. Below we give an algebraic proof of the following lemma. We give a combinatorial interpretation of \eqref{4.1.2} at the end of the section.
 \begin{lem}
 \label{lem:interchange}
   Let $\underline{Y}=(y_1,\dots,y_{m-1})$ { and let $\lambda$ be a partition of length at most $n$}. Then 
\begin{equation}
\label{4.1.1}
     \begin{split}
    \sum_{\lambda/\nu \, \in \,  \VS} 
\left( 
\sum_{
\nu/\zeta \, \in \, \BS(r)} (-1)^{\hgt(\nu/\zeta)}
\spo_{\zeta}(X/\underline{Y}) 
\right) 
& y_m^{|\lambda|-|\nu|} \\
= \sum_{
\lambda/\xi \, \in \,  \BS(r)} 
(-1)^{\hgt(\lambda/\xi)} &
\left( \sum_{
\xi/\Omega \, \in \, \VS} 
\spo_{\Omega}(X/\underline{Y}) 
y_m^{|\xi|-|\Omega|}\right).  
     \end{split}
 \end{equation}
 \end{lem}
 \begin{proof}[Algebraic proof of \cref{lem:interchange}]
Since both sides of \eqref{4.1.1} can be considered as polynomials in $y_m$, {with coefficients given by vertical strips of fixed length}, it is sufficient to prove the following:
\begin{equation}
\label{4.1.2}
     \begin{split}
    \sum_{\lambda/\nu \, \in \,  \VS(s)} 
\left( 
\sum_{\nu/\zeta \, \in \, \BS(r)} (-1)^{\hgt(\nu/\zeta)}
\spo_{\zeta}(X/\underline{Y}) 
\right) 
 & 
\\
= \sum_{
\lambda/\xi \, \in \,  \BS(r)} 
(-1)^{\hgt(\lambda/\xi)} &
\left( \sum_{
\xi/\Omega \, \in \, \VS(s)} 
\spo_{\Omega}(X/\underline{Y}) \right).  
     \end{split}
 \end{equation}
for all $s \geq 0$. Since $p_r^{\perp}   e_s^{\perp} s_{\lambda}=e_s^{\perp}  p_r^{\perp} s_{\lambda},$ by \eqref{pperp} and \eqref{eperp}, 
 we have
 \[
  p_r^{\perp} \left(
 \sum_{\lambda/\nu \, \in \, \VS(s)}  s_{\nu}
 \right)
 =
 e_s^{\perp} \left(  \sum_{\lambda/\xi \, \in \, \BS(r)}  (-1)^{\hgt(\lambda/\xi)} s_{\xi} \right) 
  {\mbox{ and }}
 \]
 \begin{equation}
 \label{4.1.3}
 \sum_{\lambda/\nu \, \in \, \VS(s)}  
 \left(
 \sum_{\nu/\zeta \, \in \, \BS(r)}  (-1)^{\hgt(\nu/\zeta)} s_{\zeta}
 \right)
 =
 \sum_{\lambda/\xi \, \in \, \BS(r)}  (-1)^{\hgt(\lambda/\xi)} \left(  \sum_{\xi/\Omega \, \in \, \VS(s)}  s_{\Omega} \right).
 \end{equation}
 Note that \eqref{4.1.3} is the same as \eqref{4.1.2} if we replace $s_{\gamma}$ by $\spo_{\gamma}(X/\underline{Y})$ for all partitions $\gamma$. 
 Since $\{s_{\lambda}(X) \mid 
 \ell(\lambda) \leq n\}$ forms a basis for the ring of symmetric functions, the partitions involved on both sides (after cancellations, if any) of the above equality are the same. {It follows that} the equality in \eqref{4.1.1} also holds. This completes the proof.
 \end{proof}
  \begin{lem} 
  \label{lem:add}
  Let $\underline{Y}=(y_1,\dots,y_{m-1})$ { and let $\lambda$ be a partition of length at most $n$. Then }
 \begin{equation}
 \label{4.2.3}
     \begin{split}
        & \sum_{
         \lambda/\nu \, \in \,  \VS} 
         \sum_{
         \sigma/\nu \, \in \, \BS(r)} (-1)^{\hgt(\sigma/\nu)}
\spo_{\sigma}(X/\underline{Y}) 
y_m^{|\lambda|-|\nu|}
+
(-1)^{r-1} \sum_{\lambda/\nu \, \in \, \VS}  
\spo_{\nu}(X/\underline{Y}) y_m^{|\lambda|-|\nu|+r}
\\
 & \qquad \qquad \qquad \qquad = \sum_{
 \eta/\lambda \, \in \, \BS(r)}
(-1)^{\hgt(\eta/\lambda)} \spo_{\eta}(X/Y).
     \end{split}
 \end{equation}   
 \end{lem}
 \begin{proof}
By \cref{schurmn}, we have 
\begin{equation}
\label{4.2.1}
P_r(X/Y)\hs_{\lambda}(X/Y)=\sum_{\eta/\lambda \, \in\, \BS(r)} (-1)^{\hgt(\eta/\lambda)} \hs_{\eta}(X/Y).
\end{equation}
Since 
 \[
 P_r(X/Y) = P_{r}(X/\underline{Y})+(-1)^{r-1} y_m^r
 \text{ and } 
 \hs_{\lambda}(X/Y)=\sum_{\lambda/\nu \, \in \, \VS}  \hs_{\nu}(X/\underline{Y}) y_m^{|\lambda|-|\nu|},
\]
the equality in \eqref{4.2.1} is the same as
\begin{equation}
\begin{split}
\sum_{\lambda/\nu \, \in \, \VS}  P_{r}(X/\underline{Y}) & \hs_{\nu}(X/\underline{Y}) y_m^{|\lambda|-|\nu|} + (-1)^{r-1} \sum_{\lambda/\nu \, \in \, \VS}  \hs_{\nu}(X/\underline{Y}) y_m^{|\lambda|-|\nu|+r}
\\&
=\sum_{\eta/\lambda \, \in\, \BS(r)} (-1)^{\hgt(\eta/\lambda)} \hs_{\eta}(X/Y).
\end{split}
\end{equation} 
Applying \cref{thm:hook}, we have 
 \begin{equation}
 \label{4.2.2}
     \begin{split}
         \sum_{
         \lambda/\nu \, \in \,  \VS} & \left(\sum_{
         \sigma/\nu \, \in \, \BS(r)} (-1)^{\hgt(\sigma/\nu)}
\hs_{\sigma}(X/\underline{Y}) \right) y_m^{|\lambda|-|\nu|}
+
(-1)^{r-1} \sum_{\lambda/\nu \, \in \, \VS}  
\hs_{\nu}(X/\underline{Y}) y_m^{|\lambda|-|\nu|+r}
\\
 & \qquad \qquad \qquad
 =\sum_{\eta/\lambda \, \in\, \BS(r)} (-1)^{\hgt(\eta/\lambda)} \hs_{\eta}(X/Y).
     \end{split}
 \end{equation}   
Note that the equality in \eqref{4.2.2} is the same as the equality in \eqref{4.2.3} if we replace $\hs_{\gamma}(X/\underline{Y})$ by $\spo_{\gamma}(X/\underline{Y})$ for all partitions $\gamma$. 
Since $\{\hs_{\lambda}(X/Y) \mid 
\lambda_{n+1} \leq m\}$ forms a basis for the ring of supersymmetric functions, the partitions involved on both sides (after cancellations if any) of the above equation are same. So, the {equality in \eqref{4.1.3}} also holds. This completes the proof.
 \end{proof}

\bigskip
{We now prove an orthosymplectic Murnaghan--Nakayama rule.}

\begin{thm} 
\label{thm:osmp-m}
Let $\lambda$ be a partition of length at most $n$ and $r \geq 1$  an integer. Then
 \begin{equation*}
 \label{spo-MN-m}
 \begin{split}
     P_r(X,\X/Y) \spo_{\lambda}(X/Y)
       =\sum_{\eta/\lambda \, \in \, \BS(r)} (-1)^{\hgt(\eta/\lambda)} & \spo_{\eta}(X/Y)
       + \sum_{\lambda/\xi \, \in \,  \BS(r)} 
       (-1)^{\hgt(\lambda/\xi)} \spo_{\xi}(X/Y)
       \\
     + \sum_{{\substack{\mu \subset \lambda\\  \mu_n<r-1}}} 
   \Bigg( &
\sum_{q=m(\mu)+1}^n {(-1)^{p(q)-q+1}}
\sp_{\mu^{(q)}}(X) \Bigg)  
\s_{\lambda'/\mu'}(Y).
 \end{split}
 \end{equation*}
 \end{thm}

 \begin{proof}[Proof of \cref{thm:osmp-m}]
  We proceed by induction on $m$, {the number of variables in $Y$}. If $m=0$, then {the result} holds by \eqref{sp-MN}.  Assume the result holds if $Y=(y_1,\dots,y_b)$ for $b<m$. Let $Y=(y_1,\dots,y_m)$ and $\underline{Y}=(y_1,\dots,y_{m-1})$. Then by \eqref{sporec},  
     \begin{equation}
     \label{3}
\spo_{\lambda}(X/Y)=\sum_{\lambda/\nu \, \in \, \VS}  \spo_{\nu}(X/\underline{Y}) y_m^{|\lambda|-|\nu|}.
\end{equation}
Since $P_r(X,\X/Y)=P_r(X,\X/\underline{Y})+(-1)^{r-1} y_m^{r}$, we have  
\begin{equation}
\label{2}
    \begin{split}
        P_r(X,\X&/Y)\spo_{\lambda}(X/Y)
        \\
         = & \sum_{
         \lambda/\nu \, \in \, \VS}  P_r(X,\X/\underline{Y}) \spo_{\nu}(X/\underline{Y}) y_m^{|\lambda|-|\nu|}
+
(-1)^{r-1} \sum_{
\lambda/\nu \, \in \, \VS}  
\spo_{\nu}(X/\underline{Y}) y_m^{|\lambda|-|\nu|+r}.
    \end{split}
\end{equation}

Now consider the first sum on the right-hand side of \eqref{2}. Using the induction hypothesis, 
\begin{equation*}
    \begin{split}
 & \sum_{
 \lambda/\nu \, \in \,  \VS}
    \left( \sum_{
    \sigma/\nu \, \in \, \BS(r)} (-1)^{\hgt(\sigma/\nu)}
\spo_{\sigma}(X/\underline{Y}) 
+
\sum_{
\nu/\zeta \, \in \, \BS(r)} (-1)^{\hgt(\nu/\zeta)}
\spo_{\zeta}(X/\underline{Y})
\right) y_m^{|\lambda|-|\nu|}\\
 & \qquad \qquad 
 +
 {\sum_{
 \lambda/\nu \, \in \,  \VS} 
 \left(  
 \sum_{\substack{\mu \subset \nu \\
 \mu_n < r-1}}
 \left(
 \sum_{q=m(\mu)+1}^n (-1)^{p(q)-q+1} \sp_{\mu^{(q)}} (X) \right)
 s_{\nu'/\mu'}(\underline{Y})
 \right)  y_m^{|\lambda|-|\nu|}}
    \end{split}
\end{equation*}

\begin{equation*}
    \begin{split}
= & \sum_{
 \lambda/\nu \, \in \,  \VS}
    \left( \sum_{
    \sigma/\nu \, \in \, \BS(r)} (-1)^{\hgt(\sigma/\nu)}
\spo_{\sigma}(X/\underline{Y}) 
+
\sum_{
\nu/\zeta \, \in \, \BS(r)} (-1)^{\hgt(\nu/\zeta)}
\spo_{\zeta}(X/\underline{Y})
\right) y_m^{|\lambda|-|\nu|}\\
 & \qquad \qquad 
 +
 {\sum_{\substack{\mu \subset \lambda \\
 \mu_n < r-1}}
 \left(
 \sum_{q=m(\mu)+1}^n (-1)^{p(q)-q+1} \sp_{\mu^{(q)}} (X) \right)
 s_{\lambda'/\mu'}(Y)}
    \end{split}
\end{equation*}

\begin{equation*}
    \begin{split}
= & \sum_{
 \lambda/\nu \, \in \,  \VS}
    \left( \sum_{
    \sigma/\nu \, \in \, \BS(r)} (-1)^{\hgt(\sigma/\nu)}
\spo_{\sigma}(X/\underline{Y}) 
\right) y_m^{|\lambda|-|\nu|} 
+ \sum_{
\lambda/\xi \, \in \, \BS(r)} (-1)^{\hgt(\lambda/\xi)}
\spo_{\xi}(X/Y)
\\
 & \qquad \qquad 
 +
  \sum_{\substack{\mu \subset \lambda \\
 \mu_n < r-1}}
 \left(
 \sum_{q=m(\mu)+1}^n (-1)^{p(q)-q+1} \sp_{\mu^{(q)}} (X) \right)
 s_{\lambda'/\mu'}(Y),
    \end{split}
\end{equation*}
where the first equality uses \eqref{ss-rec} and we use \cref{lem:interchange} and \eqref{3} to get the last equality. Adding the second sum from the right-hand side of \eqref{2}
and then applying \cref{lem:add} completes the proof.
 \end{proof}

\bigskip
 
The combinatorial procedure described above \cref{eg:comb} can be used to find the partitions $\mu^{(q)}$ in the third summand.

{ \begin{eg}
\label{eg:spo}
Let
$\lambda=(2,2), n=2, m=2$ and $r=3$. Four partitions yield non-zero $\mu^{(q)}$ in the third summand:  $\mu=(1)$ gives $\mu^{(2)}=(1,1)$; $\mu=(2)$ gives $\mu^{(2)}=(2,1)$, $\mu=(1,1)$ gives $\mu^{(2)}=(1)$ and $\mu=(2,1)$ gives $\mu^{(2)}=(2).$  We have
\begin{equation*}\begin{split}P_3(X, \X/Y)\spo_{(2,2)}(X/Y) =& \spo_{(5,2)}-\spo_{(4,3)}-\spo_{(2,2,2,1)}
+\spo_{(2,2,1,1,1)}\\
&- \sp_{(1,1)}(X)\s_{(2,2)/(1)}(Y)
 -\sp_{(2,1)}(X)\s_{(2,2)/(1,1)}(Y)\\
 & -\sp_{(1)}(X)\s_{(2,2)/(2)}(Y)
 -\sp_{(2)}(X)\s_{(2,2)/(2,1)}(Y).
\end{split}\end{equation*}
 \end{eg}}

\bigskip
\begin{rem}
{Note that it is not always possible to express the third summand as a linear combination of orthosymplectic Schur polynomials with the same number of variables. 
For instance, if $\lambda=(1)$, $n=2$, $m=1$ and $r=3$, we have
$$P_3(X, \X/Y)\spo_{(1)}(X/Y)=\spo_{(4)}(X/Y)-\spo_{(2,2)}(X/Y)+\spo_{(1,1,1,1)}(X/Y)-\sp_{(1,1)}(X).$$}
\end{rem}

\bigskip

\begin{cor}
    Let $\lambda$ be a partition of length at most $n$ such that $\lambda_n+1 \geq m+r$.
    Then
 \begin{equation*}
     P_r(X,\X/Y) \spo_{\lambda}(X/Y)
       =\sum_{\eta/\lambda \, \in \, \BS(r)} (-1)^{\hgt(\eta/\lambda)} \spo_{\eta}(X/Y)
       + \sum_{\lambda/\xi \, \in \,  \BS(r)} 
       (-1)^{\hgt(\lambda/\xi)} \spo_{\xi}(X/Y).
 \end{equation*}
\end{cor}
 

Now we write a combinatorial proof of the equality given in \eqref{4.1.2}.

\begin{proof}[Combinatorial proof of \eqref{4.1.2}] Recall the statement of \eqref{4.1.2}, which is as follows: 

 \begin{equation*}
 \label{switch}
     \begin{split}
    \sum_{\lambda/\nu \, \in \,  \VS(s)} 
\left( 
\sum_{\nu/\zeta \, \in \, \BS(r)} (-1)^{\hgt(\nu/\zeta)}
\spo_{\zeta}(X/\underline{Y}) 
\right) 
 & 
\\
= \sum_{
\lambda/\xi \, \in \,  \BS(r)} 
(-1)^{\hgt(\lambda/\xi)} &
\left( \sum_{
\xi/\Omega \, \in \, \VS(s)} 
\spo_{\Omega}(X/\underline{Y}) \right).   
     \end{split}
 \end{equation*}

We will show that if $\Omega$ can be obtained by first removing a length $r$ border strip from $\lambda$, followed by a vertical strip of length $s$, then either $\Omega$ results in a cancellation on the right-hand side of \eqref{4.1.2}  or $\Omega$ can also be obtained from first removing a length $s$ vertical strip from $\lambda$ followed by a length $r$ border strip. In the illustrative figures, border strips are marked in red and vertical strips are marked in blue.

Let $\lambda=(\lambda_1,\dots,\lambda_n)$ and $\lambda/\xi \in \BS(r)$. Then 
 \[
\xi = (\lambda_1,\dots,\lambda_{q-1},{\lambda_{q+1}-1}
,\dots,\lambda_p-1,{\lambda_q+p-q-r}
,\lambda_{p+1},\dots,\lambda_n),
 \]
for some $q \leq p$. Here $q$ is the highest row that contains a box in the border strip, $p$ is the lowest row that contains a box in the border strip and $\hgt(\lambda/\xi)=p-q$. Let $\xi/\Omega \in \VS(s)$.
Then there exist  $i_1,\dots,i_s$ such that
 \[
 \Omega_i=\begin{cases}
   \xi_i-1  & \text{ if } i \in \{i_1,\dots,i_s\}\\
     \xi_i & \text{ otherwise}.
 \end{cases}
 \]
These are the rows that contain boxes from the vertical strip.  

\vspace{.1 in}

\noindent \underline{Case 1 a)} Suppose $p \in \{i_1,\dots,i_s\}$, $\lambda_q=\lambda_{q+1}$, and 
$q \not \in \{i_1,\dots,i_s\}$, as in the left-hand configuration in 
\cref{fig:cancconfig}.  Consider
\[
\kappa=(\lambda_1,\dots,\lambda_{q-1}, \lambda_q,\underbrace{\lambda_{q+2}-1}_{\tiny \mbox{position } q+1},\dots,\lambda_p-1,\underbrace{\lambda_{q+1}+p-q-1-r}_{\tiny \mbox{position } p},\lambda_{p+1},\dots,\lambda_n).
\] We have $\lambda/\kappa \in \BS(r)$, but $\hgt(\lambda/\kappa)=p-q-1$.  Removing a vertical strip from $\kappa$
corresponding to rows $\{i_1,\dots,i_s,q\}/\{p\}$ also gives $\Omega$, as illustrated  below.  
\begin{figure}[H]
\centering
\ytableausetup{notabloids, smalltableaux}
\begin{ytableau}
\n\\
\n &\n & \n&\n & \n & \n & \n  & *(red) \\
\n\\
\n &\n & \n&\n & \n & \n & \n  & \n [\vdots]\\ 
\n &\n & \n&\n & \n & \n & \n  & \n\\
\n &\n & \n& \n & \n & \n & \n  & *(red) \\
 \n & \n & \n & \n & \n & \n & \n & \n \\
\n &\n & \n&\n & \n & \n  & \n [\iddots] \\ 
\n &\n & \n& \n & \n & \n & \n  & \n \\ 
*(blue) & *(red)  & \n & \n[\hdots] & \n & *(red)\\
\end{ytableau} \hspace{1.25cm}\raisebox{-.35in}{$\xleftrightarrow{\hspace{1cm}}$}
\ytableausetup{notabloids, smalltableaux}
\begin{ytableau}
\n\\
\n &\n & \n&\n & \n & \n & \n  & *(blue) \\
\n\\
\n &\n & \n&\n & \n & \n & \n  & \n [\vdots]\\ 
\n &\n & \n&\n & \n & \n & \n  & \n\\
\n &\n & \n& \n & \n & \n & \n  & *(red) \\
 \n & \n & \n & \n & \n & \n & \n & \n \\
\n &\n & \n&\n & \n & \n  & \n [\iddots] \\ 
\n &\n & \n& \n & \n & \n & \n  & \n \\ 
*(red) & *(red)  & \n & \n[\hdots] & \n & *(red)\\
\end{ytableau} 

\vspace{.4cm}
\caption{Cancellation configuration}
\label{fig:cancconfig}
\end{figure}
\noindent Since the attached signs are opposite, this case corresponds to a cancellation in the right-hand side of \eqref{4.1.2}. 

\vspace{.1 in}

\noindent \underline{Case 1 b)}  If
 $p \in \{i_1,\dots,i_s\}$ 
and $\lambda_{q} > \lambda_{q+1}$ or $q \in \{i_1,\dots,i_s\}$, 
consider $\nu$ defined as follows:  

 $$\nu_i = \begin{cases}
 \lambda_i-1 & \text{ if } i=q \\
     \lambda_{i}-1 & \text{ if } i \in \{i_1,\dots,i_s\} \cap ([1,q-1] \cup [p+1,n]) \\ 
     \lambda_{i}-1 &  \text{ if } i  \in \{i_1+1,\dots,i_s+1\} \cap [q+1,p] \\
   \lambda_{i} & \text{ otherwise}.
 \end{cases}$$  
Then $\lambda/\nu \in \VS(s)$, $\Omega \subset \nu$ and 
$\nu/\Omega \in \BS(r)$ with $\hgt(\nu/\Omega)=p-q$. $($See \cref{fig:config2} below, where $\lambda_q > \lambda_{q+1}$. A similar figure can be drawn when $q \in \{i_1,\dots,i_s\}.)$ Therefore, $\Omega$ also appears in the left-hand side of \eqref{4.1.2} with the same sign.

\begin{figure}[H]
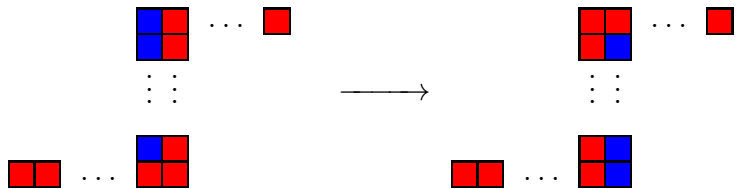

\centering
\ytableausetup{notabloids, smalltableaux}
\begin{ytableau}
\n &\n & \n&\n & \n & \n & \n  & *(red) & \n & \n [\hdots] & \n &  *(red) & *(red)\\
\n\\
\n &\n & \n&\n & \n & \n & \n  & \n [\vdots]\\ 
\n &\n & \n&\n & \n & \n & \n  & \n\\
\n &\n & \n& \n & \n & \n & \n  & *(red) \\
 \n & \n & \n & \n & \n & \n & \n & \n \\
\n &\n & \n&\n & \n & \n  & \n [\iddots] \\ 
\n &\n & \n& \n & \n & \n & \n  & \n \\ 
*(blue) & *(red)  & \n & \n[\hdots] & \n & *(red)\\
*(blue) & *(red) \\
\n   \\
 \n [\vdots] & \n [\vdots]\\
\n & \n \\
*(blue) & *(red)  \\
\end{ytableau}  
\hspace{.5cm}\raisebox{-.5in}{$\xrightarrow{\hspace{1cm}}$} 
\ytableausetup{notabloids, smalltableaux} 
  \begin{ytableau}
\n &\n & \n&\n & \n & \n & \n  & *(red) & \n & \n [\hdots] & \n &  *(red) & *(blue)\\
\n \\
\n &\n & \n&\n & \n & \n & \n  & \n [\vdots]\\ 
\n &\n & \n&\n & \n & \n & \n  & \n\\
\n &\n & \n& \n & \n & \n & \n  & *(red) \\
\n \\
\n &\n & \n&\n & \n & \n  & \n [\iddots] \\ 
\n &\n & \n& \n & \n & \n & \n  & \n \\ 
*(red) & *(red)  & \n & \n[\hdots] & \n & *(red)\\
*(red) & *(blue) \\
\n \\
\n [\vdots] & \n [\vdots]\\
\n & \n \\
*(red) & *(blue)  \\
\end{ytableau}  
\vspace{0.3in}
\caption{Configuration 1}
\label{fig:config2}
\end{figure}


\noindent \underline{Case 2 a)}  
Suppose $p \not \in \{i_1,\dots,i_s\}$ and $q-1 \not \in \{i_1,\dots,i_s\}$.

\vspace{.1 in}
\noindent Consider 
a partition $\sigma$ as follows, which is illustrated in \cref{fig:config1}:
 \[
 \sigma_i = \begin{cases}
     \lambda_{i}-1 & \text{ if } i \in \{i_1,\dots,i_s\} \cap ([1,q-2] \cup [p+1,n]) \\ 
     \lambda_{i}-1 &  \text{ if } i  \in \{i_1+1,\dots,i_s+1\} \cap [q+1,p]\\
   \lambda_{i} & \text{ otherwise}.
 \end{cases}
 \]

\begin{figure}[H]
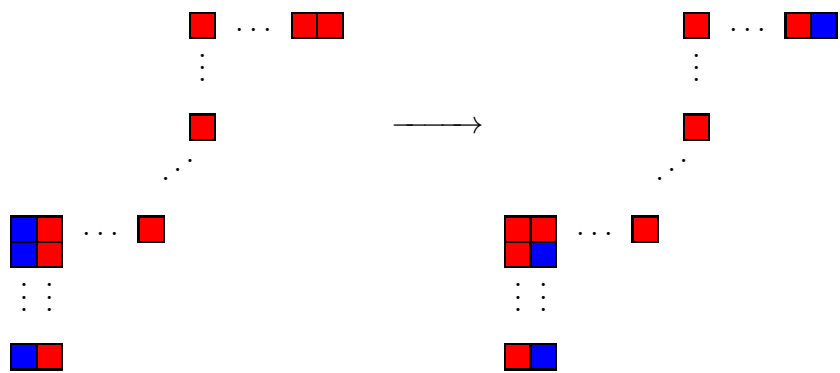

  \centering
    \ytableausetup{notabloids, smalltableaux}
  \begin{ytableau}               
\n & \n & \n & \n & \n &    *(blue) & *(red)  & \n & \n[\hdots] & \n & *(red)\\
\n & \n & \n & \n & \n &       *(blue) & *(red) \\ 
\n \\ 
\n &\n & \n & \n & \n & \n [\vdots] & \n [\vdots]\\
 \n &\n& \n & \n & \n & \n &       \n & \n \\
\n & \n &  \n & \n    &   \n &*(blue) & *(red)  \\
 *(red)  &  *(red) & \n & \n [\hdots]        & \n &*(red) & *(red)\\
\end{ytableau} 
\hspace{.5cm}\raisebox{-.35in}{$\xrightarrow{\hspace{1cm}}$} 
\ytableausetup{notabloids, smalltableaux}
  \begin{ytableau}               
    \n & \n & \n & \n & \n &*(red) & *(red)  & \n & \n[\hdots] & \n & *(red)\\
\n & \n & \n & \n & \n &       *(red) & *(blue) \\  
\n \\
\n &\n & \n & \n & \n &  \n [\vdots] & \n [\vdots]\\
\n &\n& \n & \n & \n & \n &  \n & \n \\
\n & \n &  \n & \n    &   \n &*(red) & *(blue)  \\
*(red)  &  *(red) & \n & \n [\hdots] & \n &*(red) & *(blue)\\
\end{ytableau}
\vspace{0.3in}
\caption{Configuration 2}
\label{fig:config1}
\end{figure}
Here $\lambda/\sigma \in \VS(s)$, $\Omega \subset \sigma$,
$\sigma/\Omega \in \BS(r)$ and $\hgt(\sigma/\Omega)=p-q$.  Thus $\Omega$ also appears in the left-hand side of \eqref{4.1.2} with the same coefficient.

 \vspace{.1 in}
\noindent \underline{Case 2 b)} 
Suppose $p \not \in \{i_1,\dots,i_s\}$ and $q-1 \in \{i_1,\dots,i_s\}$. 
\vspace{.1 in}

\noindent If $\lambda_{q-1} > \lambda_q$, consider the partition $\upsilon$ below (see \cref{fig:confignew}):
  \[
 \upsilon_i = \begin{cases}
     \lambda_{i}-1 & \text{ if } i \in \{i_1,\dots,i_s\} \cap ([1,q-1] \cup [p+1,n]) \\ 
     \lambda_{i}-1 &  \text{ if } i  \in \{i_1+1,\dots,i_s+1\} \cap [q+1,p]\\
   \lambda_{i} & \text{ otherwise.}
 \end{cases}
 \]

\begin{figure}[H]
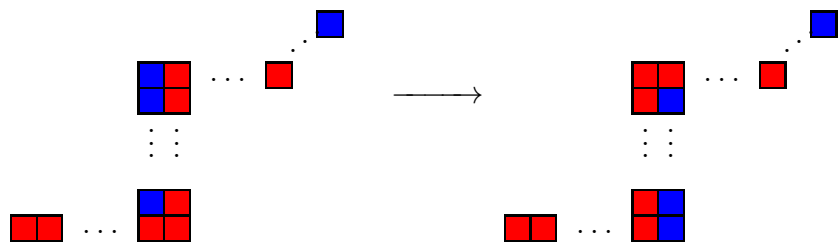

  \centering
    \ytableausetup{notabloids, smalltableaux}
  \begin{ytableau}               
\n &  \n & \n & \n & \n &   \n &\n & \n& \n & \n & \n & \n & \n[\hdots] &\n  & *(blue)\\
\n & \n & \n & \n & \n &    *(blue) & *(red)  & \n & \n[\hdots] & \n & *(red) \\
\n & \n & \n & \n & \n &       *(blue) & *(red) \\ 
\n \\ 
\n &\n & \n & \n & \n & \n [\vdots] & \n [\vdots]\\
 \n &\n& \n & \n & \n & \n &       \n & \n \\
\n & \n &  \n & \n    &   \n &*(blue) & *(red)  \\
 *(red)  &  *(red) & \n & \n [\hdots]        & \n &*(red) & *(red)\\
\end{ytableau} 
\hspace{.5cm}\raisebox{-.35in}{$\xrightarrow{\hspace{1cm}}$} 
\ytableausetup{notabloids, smalltableaux}
  \begin{ytableau}               
\n &  \n & \n & \n & \n &   \n &\n & \n& \n & \n & \n & \n & \n[\hdots] & \n & *(blue)\\
    \n & \n & \n & \n & \n &*(red) & *(red)  & \n & \n[\hdots] & \n & *(red)\\
\n & \n & \n & \n & \n &       *(red) & *(blue) \\  
\n \\
\n &\n & \n & \n & \n &  \n [\vdots] & \n [\vdots]\\
\n &\n& \n & \n & \n & \n &  \n & \n \\
\n & \n &  \n & \n    &   \n &*(red) & *(blue)  \\
*(red)  &  *(red) & \n & \n [\hdots] & \n &*(red) & *(blue)\\
\end{ytableau}
\vspace{0.3in}
\caption{Configuration 3}
\label{fig:confignew}
\end{figure}

If $\lambda_{q-1} = \lambda_q$ and $\lambda_p-\lambda_q-p+q+r=1$, let $\upsilon$ be as follows (see \cref{fig:confignew-2}):
 \[
 \upsilon_i = \begin{cases}
 \lambda_i-1 & \text{ if } i=p\\
   \lambda_{i}-1 & \text{ if } i \in \{i_1,\dots,i_s\} \cap ([1,q-2] \cup [p+1,n]) \\
 \lambda_{i}-1    &  \text{ if } i \in 
  \{i_1+1,\dots,i_s+1\}
 \cap [q+1,p],\\
   \lambda_i   & \text{ otherwise.}
 \end{cases}
 \]

\begin{figure}[H]
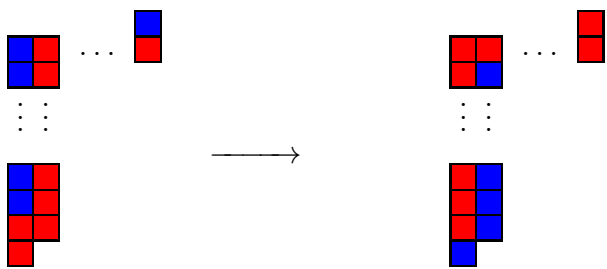

  \centering
    \ytableausetup{notabloids, smalltableaux}
  \begin{ytableau}               
\n &  \n & \n & \n & \n &   \n &\n & \n& \n & \n  & *(blue)\\ 
\n & \n & \n & \n & \n &    *(blue) & *(red)  & \n & \n[\hdots] & \n & *(red) \\
\n & \n & \n & \n & \n &       *(blue) & *(red) \\ 
\n \\ 
\n &\n & \n & \n & \n & \n [\vdots] & \n [\vdots]\\
 \n &\n& \n & \n & \n & \n &       \n & \n \\
\n & \n &  \n & \n    &   \n &*(blue) & *(red)  \\
 \n  &  \n & \n & \n     & \n &*(blue) & *(red)\\
 \n  &  \n & \n & \n     & \n &*(red) & *(red)\\
 \n  &  \n & \n & \n     & \n  & *(red)
\end{ytableau} 
\hspace{.5cm}\raisebox{-.665in}{$\xrightarrow{\hspace{1cm}}$} 
\ytableausetup{notabloids, smalltableaux}
  \begin{ytableau}               
\n &  \n & \n & \n & \n &   \n &\n & \n& \n & \n  & *(red)\\ 
\n & \n & \n & \n & \n &    *(red) & *(red)  & \n & \n[\hdots] & \n & *(red) \\
\n & \n & \n & \n & \n &       *(red) & *(blue) \\ 
\n \\ 
\n &\n & \n & \n & \n & \n [\vdots] & \n [\vdots]\\
 \n &\n& \n & \n & \n & \n &       \n & \n \\
\n & \n &  \n & \n    &   \n &*(red) & *(blue)  \\
 \n  &  \n & \n & \n     & \n &*(red) & *(blue)\\
 \n  &  \n & \n & \n     & \n &*(red) & *(blue)\\
 \n  &  \n & \n & \n     & \n  & *(blue)
\end{ytableau}
\vspace{0.3in}
\caption{Configuration 4}
\label{fig:confignew-2}
\end{figure} 

In both of the cases above illustrated by \cref{fig:confignew} and \cref{fig:confignew-2}, $\lambda/\upsilon \in \VS(s)$, $\Omega \subset \upsilon$, $\upsilon/\Omega \in \BS(r)$ and $\hgt(\lambda/\upsilon)=p-q$, so $\Omega$ appears in the left-hand side of \eqref{4.1.2} with the same sign. 

{Finally, if $\lambda_{q-1} = \lambda_q$ and $\lambda_p-\lambda_q-p+q+r>1$,
then
$\Omega$ can be obtained from 
\[
\rho = (\lambda_1,\dots,\lambda_{q-2},\lambda_{q-1}-1,\lambda_{q+1}-1,\dots,\lambda_p-1,\lambda_{q}+p-q-r+1,\lambda_{p+1},\dots,\lambda_n),
\]
by removing a vertical strip corresponding to rows $\{i_1,\dots,i_s,p\}/\{q-1\}$ (see \cref{fig:confignew-1}).

\begin{figure}[H]
  \centering
    \ytableausetup{notabloids, smalltableaux}
  \begin{ytableau}               
\n &  \n & \n & \n & \n &   \n &\n & \n& \n & \n & *(blue)\\ 
\n & \n & \n & \n & \n &    *(blue) & *(red)  & \n & \n[\hdots] & \n & *(red) \\
\n & \n & \n & \n & \n &       *(blue) & *(red) \\ 
\n \\ 
\n &\n & \n & \n & \n & \n [\vdots] & \n [\vdots]\\
 \n &\n& \n & \n & \n & \n &       \n & \n \\
\n & \n &  \n & \n    &   \n &*(blue) & *(red)  \\
 *(red)  &  *(red) & \n & \n [\hdots]        & \n &*(red) & *(red)\\
\end{ytableau} 
\hspace{.5cm}\raisebox{-.35in}{$\xleftrightarrow{\hspace{1cm}}$} 
\ytableausetup{notabloids, smalltableaux}
 \begin{ytableau}               
\n &  \n & \n & \n & \n &   \n &\n & \n& \n & \n & *(red)\\ 
\n & \n & \n & \n & \n &    *(blue) & *(red)  & \n & \n[\hdots] & \n & *(red) \\
\n & \n & \n & \n & \n &       *(blue) & *(red) \\ 
\n \\ 
\n &\n & \n & \n & \n & \n [\vdots] & \n [\vdots]\\
 \n &\n& \n & \n & \n & \n &       \n & \n \\
\n & \n &  \n & \n    &   \n &*(blue) & *(red)  \\
 *(blue)  &  *(red) & \n & \n [\hdots]        & \n &*(red) & *(red)\\
\end{ytableau} 
\vspace{0.3in}
\caption{Cancellation Configuration}
\label{fig:confignew-1}
\end{figure} 
Here $\lambda/\rho \in \BS(r)$ has height $p-q+1$, so this corresponds to a cancellation in the right side of \eqref{4.1.2}.}  One can describe similar cancellations that occur in  the left-hand side of \eqref{4.1.2} and, since each of the correspondences is reversible, this completes the proof.
\end{proof}

\begin{eg} In the two examples below, we illustrate the correspondence given above.
The diagrams on the left correspond to removing a border strip followed by a vertical strip and those on the right demonstrates how the same shape can be obtained by first removing a vertical strip followed by a border strip.

$$
\ydiagram[*(blue)]{4+1,0,2+1,2+1,1+1}
*[*(red)]
{5+2,4+2,3+2,3+1}
*[*(white)]{7,6,5,4,2}\ \
 \raisebox{-.25in}{$\xrightarrow{\hspace{1cm}}$} \ \
\ydiagram[*(red)]{4+2,4+1,2+3,2+1}
*[*(blue)]
{6+1,5+1,0,3+1,1+1}
*[*(white)]{7,6,5,4,2}\hspace{2cm} \ydiagram[*(blue)]{4+1,0,2+1,2+1}
*[*(red)]
{5+2,4+2,3+2,3+1,2+2}
*[*(white)]{7,6,5,4,4}\ \
 \raisebox{-.25in}{$\xrightarrow{\hspace{1cm}}$} \ \
\ydiagram[*(red)]{4+3,4+1,2+3,2+1,2+1}
*[*(blue)]
{0,5+1,0,3+1,3+1}
*[*(white)]{7,6,5,4,4}$$
\end{eg}

\bigskip

\noindent {\bf Acknowledgement.}  The authors wish to thank two anonymous referees who provided useful suggestions that improved the paper.
\bibliography{Bibliography}
\bibliographystyle{alpha}
\end{document}